\documentclass[a4paper,12pt]{article}
\usepackage{mathptmx}       
\usepackage{helvet}         
\usepackage{courier}        
\usepackage{type1cm}        
\usepackage{makeidx}         
\usepackage{graphicx}        
\usepackage{multicol}        
\usepackage[bottom]{footmisc}

\usepackage{amsmath}
\usepackage{amssymb}
\usepackage{amscd}
\usepackage{indentfirst}

\def\C{{\Bbb C}}
\def\1{{\bf 1}}

\usepackage{color}

\usepackage{amssymb}
\usepackage{latexsym}
\usepackage{amsmath}
\usepackage{amsfonts}
\usepackage{amstext}
\usepackage{amsthm}

\usepackage{dsfont}
\usepackage{mathtools}

\newcommand\gotH{{\mathfrak{H}}}

\newcommand\dR{{\mathbb{R}}}
\newcommand\dC{{\mathbb{C}}}

\newcommand\dN{{\mathbb{N}}}

\newcommand{\gd}{{\delta}}

\newcommand{\gG}{{\Gamma}}

\newcommand{\gl}{{\lambda}}

\newcommand{\gr}{{\rho}}

\newcommand\gt{{\tau}}

\newcommand\cI{{\mathcal{I}}}

\newcommand\cL{{\mathcal{L}}}

\newcommand{\ba}{\begin{array}}
\newcommand{\ea}{\end{array}}
\newcommand{\bea}{\begin{eqnarray}}
\newcommand{\eea}{\end{eqnarray}}
\newcommand{\bead}{\begin{eqnarray*}}
\newcommand{\eead}{\end{eqnarray*}}
\newcommand{\be}{\begin{equation}}
\newcommand{\ee}{\end{equation}}
\newcommand{\bed}{\begin{displaymath}}
\newcommand{\eed}{\end{displaymath}}

\newcommand{\bt}{\begin{theorem}}
\newcommand{\et}{\end{theorem}}
\newcommand{\bl}{\begin{lemma}}
\newcommand{\el}{\end{lemma}}
\newcommand{\bc}{\begin{corollary}}
\newcommand{\ec}{\end{corollary}}
\newcommand{\br}{\begin{remark}}
\newcommand{\er}{\end{remark}}
\newcommand{\bd}{\begin{definition}}
\newcommand{\ed}{\end{definition}}
\newcommand{\bprop}{\begin{proposition}}
\newcommand{\eprop}{\end{proposition}}
\newcommand{\bh}{\begin{hypotheses}}
\newcommand{\eh}{\end{hypotheses}}
\newcommand{\bal}{\begin{aligned}}
\newcommand{\eal}{\end{aligned}}

\newcommand{\la}{\Label}
\newcommand{\Label}{\label}

\newcommand{\slim}{\,\mbox{\rm s-}\hspace{-2pt} \lim}
\newcommand{\wlim}{\,\mbox{\rm w-}\hspace{-2pt} \lim}

\newcommand{\vertiii}[1]{{\left\vert\kern-0.25ex\left\vert\kern-0.25ex\left\vert #1
    \right\vert\kern-0.25ex\right\vert\kern-0.25ex\right\vert}}

\def\dom{{\rm dom\,}}

\def\RE{{\rm Re\,}}

\def\dom{{\rm dom\,}}

\begin{document}

\newtheorem{lemma}{Lemma}[section]
\newtheorem{theorem}[lemma]{Theorem}
\newtheorem{corollary}[lemma]{Corollary}
\newtheorem{proposition}[lemma]{Proposition}

\theoremstyle{definition}

\newtheorem{definition}[lemma]{Definition}
\newtheorem{remark}[lemma]{Remark}
\newtheorem{exam}[lemma]{Example}

\numberwithin{equation}{section}


\begin{center}
{\large\bf Notes on the Chernoff Product Formula} \\[10pt]
{\large Valentin A. Zagrebnov}  \\[5mm]

Aix-Marseille Universit\'{e}, CNRS, Centrale Marseille\\
Institut de Math\'{e}matiques de Marseille (UMR 7373)\\
Centre de Math\'{e}matiques et Informatique - Technop\^{o}le Ch\^{a}teau-Gombert\\
39, rue F. Joliot Curie, 13453 Marseille Cedex 13, France  \\
e-mail {Valentin.Zagrebnov@univ-amu.fr}

\end{center}

\vspace{5mm}

{\sc Abstract}.
We revise the strong convergent Chernoff product formula and extend it, in a Hilbert space,
to convergence in the operator-norm topology. Main results deal with the self-adjoint
Chernoff product formula. The nonself-adjoint case concerns the quasi-sectorial contractions.
\tableofcontents
\let\thefootnote\relax\footnotetext{
\begin{tabular}{@{}l}
quasi-sectorial contractions,  Trotter-Kato product formulae.
\end{tabular}}
\section{The Chernoff product formula}\la{sec:3.1}
In paper \cite{Che68}, Chernoff has proved the following Theorem:
\begin{proposition}\label{prop:3.0.0}
Let $F(t)$ be a strongly continuous function from $[0,\infty )$ to the linear contractions on a Banach
space $\mathfrak{X}$ such that $F(0)=\mathds{1}$. Suppose that the closure $C$ of the strong derivative
$F'(+0)$ is the generator of a contraction semigroup. Then $F(t/n)^{n}$, $n \in \dN $, converges for
$t\geq 0$ and $n \rightarrow \infty $ to contraction $C_0$-semigroup $\{e^{t\,C}\}_{t\geq 0}$ in the
strong operator topology.
\end{proposition}
Note, that by condition of Proposition \ref{prop:3.0.0} the operator $C(\tau): = (F(\tau) - \mathds{1})/\tau$
for $\tau > 0$ is dissipative and $F'(+0)$ is derivative $\lim_{\tau\rightarrow +0} C(\tau)$ in the
strong resolvent sense. We also recall another version of this proposition, see \cite{Che74} (Theorem 1.1).
\begin{proposition}\label{prop:3.1.0}
Let $t \mapsto F(t)$ be a measurable operator-valued function from $[0,\infty)$ to the
linear contractions on a Banach space $\mathfrak{X}$ such that $F(0)=\mathds{1}$. If
\begin{equation}\label{eq:0.1}
\slim_{\gt\to+0}(\mathds{1} - C(\gt))^{-1} = (\mathds{1} - C)^{-1} \, ,
\end{equation}
in strong operator topology, then
\be\la{eq:0.2}
\slim_{n\to+\infty}F(t/n)^n = e^{t\,C}\, ,
\ee
for all $t\geq 0$ uniformly on bounded $t$-intervals.
\end{proposition}
%
\begin{proof} The proof needs two ingredients. The first is the \textit{Trotter-Neveu-Kato} theorem
(\cite{Kat80}, Ch.IX, Theorem 2.16): the strong convergence in (\ref{eq:0.1}) yields for contractions
$\{e^{t\, C(\gt)}\}_{t\geq 0}$
\be\la{eq:0.3}
\lim_{\gt\to+0}e^{t\, C(\gt)} u = e^{t\,C} u \, ,
\ee
for all $u \in \mathfrak{X}$, locally uniformly on closed intervals $\cI \subset \dR^+_{0}$.

Now, for any  $t > 0$ and $u \in \mathfrak{X}$ we define $\{u_n :=
[\mathds{1} - C(t/n)/\sqrt{n}]^{-1} u \}_{n\geq1}$. Then by (\ref{eq:0.3})
\begin{equation}\la{eq:0.4}
\lim_{n \rightarrow \infty } u_n :=
\lim_{n \rightarrow \infty}\sqrt{n} \int_{0}^{\infty} ds \ e^{- \sqrt{n}\, s}\ e^{s\, C(t/n)}\, u = u .
\end{equation}
The second ingredient comes from the, so-called, $\sqrt{n}$-Lemma (\cite{Che68} Lemma 2). It yields the
estimate
\begin{equation}\label{eq:0.5}
\|e^{t \, C(t/n)}\, w - F(t/n)^n \, w\| \leq \sqrt{n} \ \|(F(t/n) - \mathds{1})\, w\| \, ,
\end{equation}
for any $w\in \mathfrak{X}$. Then (\ref{eq:0.4}) and (\ref{eq:0.5}) imply
\begin{eqnarray*}
&& \lim_{n \rightarrow \infty }\|e^{t \, C(t/n)}\, u - F(t/n)^n \, u\| =
\lim_{n \rightarrow \infty }\|e^{t \, C(t/n)}\, u_n - F(t/n)^n \, u_n \| \\
&& \leq \lim_{n \rightarrow \infty } t/\sqrt{n} \ \|C(t/n)\, [\mathds{1} - C(t/n)/\sqrt{n}]^{-1}\, u\| =
\lim_{n \rightarrow \infty } t\, \|u_n - u\| = 0\,.
\end{eqnarray*}
Hence, by (\ref{eq:0.3}) together with estimate
\begin{equation*}
\|e^{t\,C} u - F(t/n)^n u\|\leq \|e^{t\,C} u  - e^{t\, C(t/n)} u\| +
\|e^{t \, C(t/n)}\, u - F(t/n)^n \, u\| \, ,
\end{equation*}
for any $u\in \mathfrak{X}$, we obtain (\ref{eq:0.2}).
\end{proof}
\begin{remark}\label{rem:3.2.0}
(a) Equivalence of condition (\ref{eq:0.1}) to the existence of the strong derivative $F'(+0)$ is
a part of the \textit{Trotter-Neveu-Kato} theorem, \cite{EN00} Ch.III, Sects.4.8 and 4.9.\\
(b) For analysis of optimality, generalisation and improvement of the $\sqrt{n}$-Lemma, see \cite{Zag17}.
\end{remark}
\begin{definition}\label{def:3.3.0}
Equation (\ref{eq:0.2}) is called the \textit{Chernoff product formula}, or the {Chernoff
approximation formula}, in the strong operator topology for contraction $C_0$-semigroup
$\{e^{t\,C}\}_{t\geq 0}$.
\end{definition}

The aim of the present Notes is \textit{lifting} the strongly convergent Chernoff product formula to
formula convergent in the operator norm topology, whereas a majority of results concerns only strong
convergence, see, e.g., a detailed review \cite{But19}.
Our main results are focused on analysis of the \textit{operator-norm} convergence of the
Chernoff product formula in a Hilbert space $\gotH$, see Section \ref{sec:3.2}.
In Section \ref{sec:3.3} we establish estimates of the operator-norm \textit{rate} of convergence for
\textit{self-adjoint} Chernoff product formula.
The case of Chernoff product formula for {nonself-adjoint} \textit{quasi-sectorial} contractions
is the subject of Section \ref{sec:3.4}. The Trotter-Kato product formula is a direct application
of the Chernoff product formula, see Section \ref{sec:3.5}.

We conclude this section by the proof of our first \textit{Note}: for strongly convergent Chernoff
product formula the \textit{self-adjointness} allows to relax condition (\ref{eq:0.1}) to the weak
operator convergence. To this aim we recall a \textit{lifting} topology assertion in $\gotH \,$:
\begin{proposition}\la{lem:3.1.1+1} Let $\{u_n\}_{n\geq1}$ be a weakly convergent sequence of vectors,
$\wlim_{n\to\infty} u_n = u$, in a Hilbert space $\gotH$. If , in addition,
$\lim_{n\to\infty} \|u_n\| = \|u\|$, then $\lim_{n\to\infty} \|u_n - u\| =0$.
\end{proposition}
This assertion implies the following statement.
\bl\la{lem:3.1.3}
Let $S: \dR^+ \rightarrow \cL(\gotH)$ be a measurable family of non-negative
self-adjoint operators and let $H$ be non-negative self-adjoint operator. If the weak operator
limit
\be\la{eq:3.1.11}
\wlim_{\gt\to+0}(\gl \mathds{1} + S(\gt))^{-1} = (\gl \mathds{1} +H)^{-1}\, ,
\ee
for each $\gl > 0$, then it is also true in the strong operator topology {\rm{:}}
\be\la{eq:3.1.12}
\slim_{\gt\to+0}(\lambda \mathds{1} + S(\gt))^{-1} = (\gl \mathds{1} +H)^{-1}\, .
\ee
\el
%
\begin{proof}
From \eqref{eq:3.1.11} we get
\be\la{eq:3.1.13}
\lim_{\gt\to+0}\|(\gl \mathds{1} + S(\gt))^{-1/2}u\|^2 = \|(\gl \mathds{1} +H)^{-1/2}u\|^2 ,
\quad u \in  \gotH,
\ee
for $\gl \ge 1$. Since
\begin{equation*}
(\mathds{1} +S(\gt))^{-1/2} = \frac{1}{\pi}\int^\infty_0 d\eta \ \frac{1}{\sqrt{\eta}}\
(\eta \, \mathds{1} + \mathds{1} + S(\gt))^{-1}\, ,
\end{equation*}
\eqref{eq:3.1.11} also yields $\wlim_{\gt\to +0}(\mathds{1} + S(\gt))^{-1/2} =
(\mathds{1} +H)^{-1/2}$. This, together with \eqref{eq:3.1.13} and the lifting Proposition \ref{lem:3.1.1+1},
implies $\slim_{\gt\to +0}(\mathds{1} + S(\gt))^{-1/2} = (\mathds{1} +H)^{-1/2}$.
Then strong convergence of the product of operators yields \eqref{eq:3.1.12}.
\end{proof}
By Lemma \ref{lem:3.1.3}, the conditions of Proposition \ref{prop:3.1.0} for the
family of operators
\be\la{eq:3.1.1}
S(\gt) := \frac{\mathds{1}- F(\gt)}{\gt}\, , \quad \gt > 0 \, ,
\ee
can be reformulated in a Hilbert space $\gotH$. Then we get a stronger assertion:
%
\bt\la{cor:3.1.4}
Let $F: \dR^+ \longrightarrow \cL(\gotH)$ be a measurable
family of non-negative self-adjoint contractions such that $F(0) = \mathds{1}$ and let $H\geq 0$
be a self-adjoint operator in $\gotH$. Then
\be\la{eq:3.1.7}
\slim_{n\to+\infty}F(t/n)^n = e^{-tH}\, ,
\ee
if and only if for each $\gl > 0$ the condition
\be\la{eq:3.1.15}
\wlim_{\gt\to +0} (\gl \mathds{1} + S(\gt))^{-1} = (\gl \mathds{1} +H)^{-1}\, ,
\ee
is satisfied.
\et
%
\begin{proof}
If \eqref{eq:3.1.15} is valid, then applying Lemma \ref{lem:3.1.3} we verify \eqref{eq:3.1.12}. Using
Proposition \ref{prop:3.1.0} in a Hilbert space $\gotH$ for $C(\tau) = - S(\gt)$, $C = -H$ and
\eqref{eq:3.1.12} for $\lambda =1$, we prove the self-adjoint Chernoff product formula \eqref{eq:3.1.7}.

To prove the converse we use the representation
\be\la{eq:3.1.150}
e^{-tS(t/n)} - e^{-tH} = F(t/n)^n - e^{-tH} + e^{-tS(t/n)} - F(t/n)^n \, .
\ee
Then $\lim_{n\to+\infty}\|(F(t/n)^n - e^{-tH})u\| = 0$ by condition (\ref{eq:3.1.7}), whereas
by the spectral functional calculus for self-adjoint operator $F(\gt)$ we obtain the
\textit{operator norm} estimate
\be\la{eq:3.1.151}
\|F(t/n)^n - e^{-n(\mathds{1}- F(t/n))}\| = \left\|\int_{[0,1]} dE_{F(t/n)}(\gl) \left(\gl^n
 -e^{-n(1-\gl)}\right)\right\| \leq \frac{1}{n} \ .
\ee
Therefore, (\ref{eq:3.1.150}) yields for $\tau = t/n$ the limit
\be\la{eq:3.1.3}
\slim_{\gt\to+0}e^{-tS(\gt)} = e^{-tH} \, ,
\ee
for $t\in \dR^+$. Note that by the self-adjoint \textit{Trotter-Neveu-Kato} convergence theorem
the limit (\ref{eq:3.1.3}) is equivalent to
\be\la{eq:3.1.2}
\slim_{\gt\to+0}(\mathds{1} + S(\gt))^{-1} = (\mathds{1} +H)^{-1} \, .
\ee
Since (\ref{eq:3.1.2}) is, in turn, equivalent to
\begin{equation*}
\slim_{\gt\to +0}(\gl \mathds{1} +S(\gt))^{-1} = (\gl \mathds{1} +H)^{-1}\, ,
\end{equation*}
for $\gl > 0$, the latter yields the weak limit  \eqref{eq:3.1.15}.
\end{proof}
\begin{remark}\label{rem:3.2.1}
(a) We note that due to self-ajointness of the family $\{F(t)\}_{t\geq0}$ the estimate (\ref{eq:3.1.151})
is stronger than the $\sqrt{n}$-estimate (\ref{eq:0.5}) for general contractions.\\
(b) By definition of $\{F(t)\}_{t\geq 0}$ and by condition (\ref{eq:3.1.15})
the $C_0$-semigroup property of $\{e^{-tH}\}_{t\geq 0}$ ensures that strong limits $\slim_{t\to +0}$ of
the left- and the right-hand sides of \eqref{eq:3.1.7} are well-defined and coincide with $\mathds{1}$.
Similarly, since by (\ref{eq:3.1.3}) the limits in $\lim_{t\to+0} \slim_{\gt\to+0}e^{-tS(\gt)}$
commute, one gets that $t \in \dR^+_{0}$.
\end{remark}
%

\section{{Lifting the Chernoff product formula to operator-norm topology}}\la{sec:3.2}

The first result about lifting the Chernoff product formula to the operator-norm topology
is due to \cite{NZ99a}, Theorem 2.2. Our next \textit{Note} will be about improvement of this result
by extension to any closed $t$-interval $\cI \subset \dR^{+}_{0}$.

Similarly to the case of the strong operator topology our strategy includes the estimate of two
ingredients (\ref{eq:0.3}) and (\ref{eq:0.5}) involved into Proposition \ref{prop:3.1.0}, but now
in the \textit{operator norm} topology. Since the second ingredient has estimate (\ref{eq:3.1.151}),
it rests to find conditions for \textit{lifting} to operator norm convergence of the limit
(\ref{eq:3.1.3}) in the {Trotter-Neveu-Kato} convergence theorem.
We proceed with the following lemma.
\bl\la{lem:3.2.1}
Let $K$ and $L$ be non-negative self-adjoint operators in a Hilbert
space $\gotH$. Then
\be\la{eq:3.2.1}
\|e^{-K} - e^{-L}\| \le c \ \|(\mathds{1} +K)^{-1} - (\mathds{1} +L)^{-1}\|
\ee
with a constant $c > 0 $ independent of operators $K$ and $L$.
\el
By the Riesz-Dunford functional calculus one obtains representation:
\be\la{eq:3.2.1-1}
e^{-K} - e^{-L} = \frac{1}{2\pi i}\int_\gG dz \;
e^{z}\left((z + K)^{-1} - (z + L)^{-1}\right).
\ee
Essentially, the line of reasoning is based on straightforward estimates that use (\ref{eq:3.2.1-1})
for a given contour $\Gamma$. Then arguments show that constant $c$ is only $\Gamma$-dependent.
We skip the proof.

By virtue of (\ref{eq:3.1.151}) and Lemma \ref{lem:3.2.1} the \textit{operator-norm} Trotter-Neveu-Kato
theorem needs \textit{lifting} of the strong convergence in \eqref{eq:3.1.2} to the operator-norm
convergence.
%
\bl\la{lem:3.2.2}
Let $F: \dR^{+}_{0} \longrightarrow \cL(\gotH)$ be a measurable family of non-negative self-adjoint
contractions such that $F(0) = \mathds{1}$. Let self-adjoint family $\{S(\tau)\}_{\tau > 0}$ be
defined by \emph{(\ref{eq:3.1.1}) and (\ref{eq:3.1.2})} for non-negative self-adjoint operator
$H$ in $\gotH$. Then condition
\be\la{eq:3.2.6}
\lim_{\gt\to+0}\|(\mathds{1} +S(\gt))^{-1} - (\mathds{1} +H)^{-1}\| = 0 ,
\ee
is satisfied if and only if
\be\la{eq:3.2.7}
\lim_{\gt\to+0}\sup_{t\in\cI}\|(\mathds{1} +t S(\tau))^{-1} - (\mathds{1} +t H)^{-1}\| = 0 ,
\ee
for any closed interval $\cI \subset \dR^+$.
\el
%
\begin{proof}
A straightforward computation shows that
\begin{equation}\la{eq:3.2.71}
\begin{split}
&(\mathds{1} +tS(\gt))^{-1} - (\mathds{1} +tH)^{-1}\\
&= t(\mathds{1} +S(\gt))(\mathds{1} +tS(\gt))^{-1}[(\mathds{1} +S(\gt))^{-1} -
(\mathds{1} +H)^{-1}](\mathds{1} +H)(\mathds{1} +tH)^{-1}\, .
\end{split}
\end{equation}
Here we used that if $t > 0$ and $\gt > 0$ then for self-adjoint operator $S(\gt)$ the closure
\begin{equation*}
\overline{(\mathds{1} + t \, S(\gt))^{-1} (\mathds{1} +S(\gt))}=
(\mathds{1} + S(\gt))(\mathds{1} + t \, S(\gt))^{-1} \, .
\end{equation*}
For these values of arguments $t$ and $\tau$ we obtain estimates:
\begin{eqnarray*}
&&\|(\mathds{1} +S(\gt))(\mathds{1} + t\, S(\gt))^{-1}\| \le (1 + 2/t)\, , \\
&&\|(\mathds{1} +H)(\mathds{1} + t\, H)^{-1}\| \le (1 + 2/t)\, .
\end{eqnarray*}

If $\cI$ is a closed interval of $\dR^+$, e.g., $\cI := [a,b]$ for
$0 < a < b < \infty$, then by (\ref{eq:3.2.71})
\begin{equation}\la{eq:3.2.8}
\|(\mathds{1} +tS(\gt))^{-1} - (\mathds{1} +tH)^{-1}\|
\le {b}(1 + 2/a)^{2} \|(\mathds{1} +S(\gt))^{-1} - (\mathds{1} +H)^{-1}\|,
\end{equation}
for $t \in [a,b]$ and $\gt > 0$. By \eqref{eq:3.2.6}
the estimate \eqref{eq:3.2.8} yields \eqref{eq:3.2.7}. The converse is obvious.
%
\end{proof}

%
\bt\la{th:3.2.3}
Let $F: \dR^{+}_{0} \longrightarrow \cL(\gotH)$ be a measurable family of non-negative
self-adjoint contractions such that $F(0) = \mathds{1}$. Let self-adjoint family
$\{S(\tau)\}_{\tau > 0}$ be defined by \emph{(\ref{eq:3.1.1}) and (\ref{eq:3.1.2})} for
self-adjoint operator $H \geq 0$ in $\gotH$. Then we have
\be\la{eq:3.2.9}
\lim_{n\to \infty}\sup_{t\in\cI}\|F(t/n)^n - e^{-tH}\| = 0 \, ,
\ee
for any closed interval $\cI \subset \dR^+$, if and only if the family $\{S(\tau)\}_{\tau > 0}$
satisfies condition \eqref{eq:3.2.6}.
\et
%
\begin{proof}
For $t > 0$ and $n \geq 1$ we get estimate
\be\la{eq:3.2.10}
\|F(t/n)^n - e^{-tH}\| \le\|F(t/n)^n - e^{-tS(t/n)}\| + \|e^{-tS(t/n)} - e^{-tH}\|.
\ee
%
%
Then (\ref{eq:3.1.151}) and (\ref{eq:3.2.10}) imply
\be\la{eq:3.2.11}
\|F(t/n)^n - e^{-tH}\| \le \frac{1}{n} + \|e^{-t S(t/n)} - e^{-tH}\|, \quad t > 0,
\quad n \geq 1 \ .
\ee

Note that by Lemma \ref{lem:3.2.1} there is a constant $c > 0$ such that for $t > 0$
\be\la{eq:3.2.12}
\|e^{-tS(t/n)} - e^{-tH}\| \le c \ \|(\mathds{1} +t S(t/n))^{-1} - (\mathds{1} +tH)^{-1}\| .
\ee
 Inserting the estimate \eqref{eq:3.2.12} into \eqref{eq:3.2.11} we obtain
\be\la{eq:3.2.13}
\|F(t/n)^n - e^{-tH}\| \le \frac{1}{n} +
c \, \|(\mathds{1} +tS(t/n))^{-1} - (\mathds{1} +tH)^{-1}\|,
\ee
Then (\ref{eq:3.2.13}) and Lemma \ref{lem:3.2.2} yield \eqref{eq:3.2.9}.

Conversely, let us assume \eqref{eq:3.2.9}. For $t > 0$ and $n \geq 1$ we have estimate
\begin{equation*}
\|e^{-tS(t/n)} - e^{-tH}\| \le \|F(t/n)^n  - e^{-tH}\| +
\|F(t/n)^n - e^{-tS(t/n)}\| \ ,
\end{equation*}
and by (\ref{eq:3.1.151})
\be\la{eq:3.2.14}
\|e^{-t S(t/n)} - e^{-tH}\| \le \|F(t/n)^n - e^{-tH}\| + \frac{1}{n} \ .
\ee

Then by assumption \eqref{eq:3.2.9} for any closed interval $\cI \subset \dR^+$ the estimate
\eqref{eq:3.2.14} yields
\begin{equation*}
\lim_{n\to\infty}\sup_{t\in\cI}\|e^{-tS(t/n)} - e^{-tH}\| = 0 \, .
\end{equation*}
Hence, $\lim_{n\to\infty}\|e^{-tS(t/n)} - e^{-tH}\| = 0$ implies
$\lim_{\gt\to+0}\|e^{-tS(\gt)} - e^{-tH}\| = 0$ for any $t > 0$.

Now, using representation:
\begin{equation}\la{eq:3.2.14-2}
(\mathds{1} +S(\gt))^{-1} - (\mathds{1} +H)^{-1} =
\int^\infty_0 ds \ e^{-s}\ \big(e^{-s \, S(\gt)} - e^{-s\, H}\big) \ ,
\end{equation}
we obtain the estimate
\begin{equation}\la{eq:3.2.14-1}
 \|(\mathds{1} +  S(\gt))^{-1} - (\mathds{1} + H)^{-1}\| \le \int^\infty_0  ds
 \ e^{-s}\ \big\|e^{-s S(\gt)} - e^{-s H}\big\| \, .
\end{equation}
Let $\Phi_{\tau}(s):= e^{-s}\ \big\|e^{-s S(\gt)} - e^{-s H}\big\|$. Since $S(\gt)\geq 0$ and $H \geq 0$,
one gets $\Phi_{\tau}(s) \leq 2 \, e^{-s} \in L^{1}(\dR^+)$ and $\lim_{\gt\to+0}\Phi_{\tau}(s) = 0$.
Then $\lim_{\gt\to+0}$ in the right-hand side of (\ref{eq:3.2.14-1}) is zero by the Lebesgue
{dominated convergence} theorem, that yields \eqref{eq:3.2.6}.
\end{proof}
Extension of this statement to \textit{any} bounded interval $\cI \subset \dR^{+}_{0}$
(cf. Remark \ref{rem:3.2.1}(b)) needs a \textit{uniform} operator-norm extension of the
Trotter-Neveu-Kato theorem, that we present below.
\bt\la{th:3.2.5}
Let $F: \dR^{+}_{0} \longrightarrow \cL(\gotH)$ be a measurable family of non-negative self-adjoint
contractions such that $F(0) = \mathds{1}$. Let self-adjoint family $\{S(\tau)\}_{\tau > 0}$ be
defined by \emph{(\ref{eq:3.1.1}) and (\ref{eq:3.1.2})\,}, for self-adjoint operator $H \geq 0$ in $\gotH$.
Then the convergence
\be\la{eq:3.2.16}
\lim_{\gt\to +0}\sup_{t\in \cI}\|e^{-tS(\tau)} - e^{-tH}\| = 0 \ ,
\ee
holds for any bounded interval $\cI \subset \dR^{+}_{0}$ if and only if the condition
\be\la{eq:3.2.17}
\lim_{\gt\to +0}\sup_{t\in \cI}\|(\mathds{1} +t S(\tau))^{-1} - (\mathds{1} +tH)^{-1}\| = 0 \ ,
\ee
is valid for any bounded interval $\cI \subset \dR^{+}_{0}$.
\et
%
\begin{proof}
By conditions of theorem and by Lemma \ref{lem:3.2.1} we obtain from \eqref{eq:3.2.12} the estimate
\begin{equation*}
\sup_{t\in\cI}\|e^{-tS(\tau)} - e^{-tH}\|  \le c \ \sup_{t\in\cI}\|(\mathds{1} + t S(\tau))^{-1} -
(\mathds{1} + t H)^{-1}\| \, ,
\end{equation*}
for $\tau > 0$ and for any bounded interval $\cI \subset \dR^{+}_{0}$. This estimate and condition
\eqref{eq:3.2.17} imply the convergence in \eqref{eq:3.2.16}.

Conversely, assume \eqref{eq:3.2.16}. Note that by representation (\ref{eq:3.2.14-2}) one gets for
$t\geq 0 $:
\begin{equation*}
(\mathds{1} +tS(\tau))^{-1} - (\mathds{1} +tH)^{-1} =
\int^\infty_0 ds \ e^{-s}\big(e^{-s\,tS(\tau)} - e^{-s\,tH}\big)\, .
\end{equation*}
This yields the estimate
\begin{equation*}
\left\|(\mathds{1} +tS(\tau))^{-1} - (\mathds{1} +tH)^{-1}\right\| \le \int^\infty_0 ds \ e^{-s}
\big\|e^{-s\,tS(\tau)} - e^{-s\,tH}\big\|  ,
\end{equation*}
for $\tau > 0$ and $t \ge 0$.

Now, let $0< \varepsilon < 1$ and let $N_{\varepsilon} := - \ln(\varepsilon/2)$. Then
\begin{equation*}
\int^\infty_{N_{\varepsilon}}  ds \ e^{-s}\left\|e^{-s\,tS(\tau)} - e^{-s\,tH}\right\| \le
{\varepsilon}\, ,
\end{equation*}
for $\tau > 0$ and $t \ge 0$. Hence,
\begin{equation*}
\left\|(\mathds{1} +tS(\tau))^{-1} - (\mathds{1} +tH)^{-1}\right\| \le
\int^{N_{\varepsilon}}_0 ds \ e^{-s}\left\|e^{-stS(\tau)} - e^{-stH}\right\| + {\varepsilon} \ ,
\end{equation*}
that for any bounded interval $\cI \subset \dR^{+}_{0}$ and $\tau> 0$ yields
\begin{equation*}
\sup_{t\in\cI}\left\|(\mathds{1} +tS(\tau))^{-1} - (\mathds{1} +tH)^{-1}\right\| \le
\sup_{\ba{c} t\in\cI \wedge s\in [0,N_{\varepsilon}]\ea}\left\|e^{-s\,tS(\tau)} - e^{-s\,tH}\right\| +
{\varepsilon} \ .
\end{equation*}
Applying now \eqref{eq:3.2.16}
we obtain
\begin{equation*}
\lim_{\tau\to +0}\sup_{t\in\cI}\left\|(\mathds{1} +tS(\tau))^{-1} -
(\mathds{1} +tH)^{-1}\right\| \le {\varepsilon} \ ,
\end{equation*}
for any $\varepsilon > 0$. This completes the proof of \eqref{eq:3.2.17}.
%
\end{proof}

Now we are in position to improve Theorem \ref{th:3.2.3}. We relax the restriction to closed
intervals $\cI \subset \dR^{+}$ to condition on any bounded interval $\cI \subset \dR^{+}_{0}$.
%
\bt\la{th:3.2.6}
Let $F: \dR^{+}_{0} \longrightarrow \cL(\gotH)$ be a measurable family of non-negative
self-adjoint contractions such that $F(0) = \mathds{1}$. Let self-adjoint family
$\{S(\tau)\}_{\tau > 0}$ be defined by \emph{(\ref{eq:3.1.1})} and \emph{(\ref{eq:3.1.2})\,}
for self-adjoint operator $H \geq 0$ in $\gotH$. Then
\be\la{eq:3.2.18}
\lim_{n\to \infty}\sup_{t\in\cI}\|F(t/n)^n - e^{-tH}\| = 0  ,
\ee
for any bounded interval $\cI \subset \dR^{+}_{0}$ if and only if
\be\la{eq:3.2.19}
\lim_{n\to\infty}\sup_{t\in\cI}\|(\mathds{1} +t S(t/n))^{-1} - (\mathds{1} +tH)^{-1}\| = 0  ,
\ee
is satisfied for any bounded interval $\cI \subset \dR^{+}_{0}$.
\et
%
\begin{proof}
By \eqref{eq:3.2.13} and by assumption \eqref{eq:3.2.19} we obtain the limit \eqref{eq:3.2.18}.

Conversely, using \eqref{eq:3.2.14} and assumption \eqref{eq:3.2.18} one gets \eqref{eq:3.2.16}
for $\tau = t/n$ and for any bounded interval $\cI \subset \dR^{+}_{0}$. Then application of
Theorem \ref{th:3.2.5} yields \eqref{eq:3.2.19}.
%
\end{proof}

\section{{Chernoff product formula: rate of the operator-norm convergence}}\la{sec:3.3}
Theorem \ref{th:3.2.5} admits further improvements. They allow to establish estimates for the
\textit{rate} of operator-norm convergence in \eqref{eq:3.2.16} under certain conditions in
\eqref{eq:3.2.17}. Our next \textit{Note} concerns the estimates of the convergence rate in Theorem
\ref{th:3.2.6}. Recall that the first result in this direction was due to \cite{IT01} (Lemma 2.1).
\bl\la{lem:3.3.1}
Let $F: \dR^{+}_{0} \longrightarrow \cL(\gotH)$ be a measurable family of non-negative
self-adjoint contractions such that $F(0) = \mathds{1}$. Let self-adjoint family
$\{S(\tau)\}_{\tau > 0}$ be defined by \emph{(\ref{eq:3.1.1}) and (\ref{eq:3.1.2})}
for self-adjoint operator $H \geq 0$ in $\gotH$.
\item[\;\;\rm (i)]
If $\gr \in (0,1]$ and there is a constant $M_\gr > 0$ such that the estimate
\be\la{eq:3.3.1}
\|(\mathds{1} +tS(\gt))^{-1} - (\mathds{1} +tH)^{-1}\| \le
M_\gr\left(\frac{\gt}{t}\right)^\gr \, ,
\ee
holds for $\gt,t \in (0,1]$ and $0 < \gt \le t$, then there is a
constant $c_\gr > 0$ such that the estimate
\be\la{eq:3.3.2}
\|F(\gt)^{t/\gt} - e^{-tH}\| \le c_\gr\left(\frac{\gt}{t}\right)^\gr \, ,
\ee
is valid for $\gt,t \in (0,1]$ with $0 < \gt \le t$.

\item[\;\;\rm (ii)]
If $\gr \in (0,1)$ and there is a constant $c_\gr$ such that
\eqref{eq:3.3.2} holds, then there is a constant $M_\gr > 0$ such that
the estimate \eqref{eq:3.3.1} is valid.
\el
%
\begin{proof}
(i) By Lemma \ref{lem:3.2.1} there is a constant $c > 0$ such that
\be\la{eq:3.3.3}
\|e^{-tS(\gt)} - e^{-tH}\| \le c \, \|(\mathds{1} +tS(\gt))^{-1} - (\mathds{1} +tH)^{-1}\| \, ,
\ee
for $\gt,t > 0$. Then \eqref{eq:3.3.1}, for $\gt,t \in (0,1]$ with $0 < \gt \le t$, yields,
cf. Theorem \ref{th:3.2.5},
\bed
\big\|e^{-tS(\gt)} - e^{-tH}\big\| \le c \;M_\gr\left(\frac{\gt}{t}\right)^\gr \, .
\eed
By definition \eqref{eq:3.1.1} and  inequality \eqref{eq:3.1.151}
\be\la{eq:3.3.5}
\big\|F(\gt)^{t/\gt} - e^{-tS(\gt)}\big\| \le \frac{\gt}{t} \, .
\ee
Therefore, by estimate
\be\la{eq:3.3.6}
\big\|F(\gt)^{t/\gt} - e^{-tH}\big\| \le
\big\|F(\gt)^{t/\gt} - e^{-tS(\gt)}\big\| + \big\|e^{-tS(\gt)} - e^{-tH}\big\| \, ,
\ee
we obtain for $\gt,t \in (0,1]$ with $0 < \gt \le t$
\bed
\big\|F(\gt)^{t/\gt} - e^{-tH}\big\| \le
\frac{\gt}{t} + c \; M_\gr\left(\frac{\gt}{t}\right)^\gr \, .
\eed
Since for $\gr \in (0,1]$ one has ${\gt}/{t} \le \left({\gt}/{t}\right)^\gr$,
\be\la{eq:3.3.8}
\left\|F(\gt)^{t/\gt} - e^{-tH}\right\| \le
(1 + c \;M_\gr)\left(\frac{\gt}{t}\right)^\gr \ .
\ee
Setting $c_\gr := 1 + c\;M_\gr \,$, we prove \eqref{eq:3.3.2} for $\gr \in (0,1]$.

(ii) To prove (\ref{eq:3.3.1}) we use the identity:
\begin{equation*}
(\mathds{1} +tS(\gt))^{-1} - (\mathds{1} +tH)^{-1} =
\sum^\infty_{n=0}\;\int^{n+1}_n \; dx \;e^{-x}
\big(e^{-xtS(\gt)} - e^{-xtH}\big)\, ,
\end{equation*}
here $\gt,t > 0$. Substitution $x = y + n$ yields
\bead
\lefteqn{
(\mathds{1} +tS(\gt))^{-1} - (\mathds{1} +tH)^{-1} =}\\
& & \hspace{1.0cm}
\sum^\infty_{n=0}\;e^{-n}\;\int^{1}_0 \; dy \;e^{-y}
\big(e^{-(y+n)tS(\gt)} - e^{-(y+n)tH}\big)\, ,
\eead
%
%
%
%
%
%
%
%
%
%
that gives the representation
\bead
\lefteqn{
(\mathds{1} +tS(\gt))^{-1} - (\mathds{1} +tH)^{-1} =}\\
& & \hspace{0.5cm}
\sum^\infty_{n=0}\;e^{-n}\left \{
\Big(\sum^{n-1}_{k=0} e^{-ktS(\gt)}\big(e^{-tS(\gt)} - e^{-tH}\big) e^{-(n-k-1)tH}\Big)
\int^{1}_0 \; dy \;e^{-y}\;e^{-ytS(\gt)} + \right.\\
& & \hspace{0.5cm}
\left.
e^{-ntH}\int^{1}_0 \; dy\;e^{-y}\;\big(e^{-ytS(\gt)} - e^{-ytH}\big)
\right\}.
\eead
Hence, we obtain the estimate
\bea\la{eq:3.3.9}
\lefteqn{
\|(\mathds{1} +tS(\gt))^{-1} - (\mathds{1} +tH)^{-1}\| \le}\\
& & \hspace{0.5cm}
\sum^\infty_{n=0}\;e^{-n}\Big\{n \ \big\|e^{-tS(\gt)} - e^{-tH}\big\| +
\int^{1}_0 \; dy \;e^{-y}\;\big\|e^{-ytS(\gt)} - e^{-ytH}\big\|\Big\}\, .
\nonumber
\eea
%
%

Note that assumption (\ref{eq:3.3.2}) and estimate \eqref{eq:3.3.5} yield
for $\gt,t \in (0,1]$, with $0 < \gt \le t$,
\be\la{eq:3.3.10}
\big\|e^{-tS(\gt)} - e^{-tH}\big\| \le {{(1 + c_\gr)}} \left(\frac{\gt}{t}\right)^\gr .
\ee
To treat the last term in \eqref{eq:3.3.9}
we use decomposition
\bea\la{eq:3.3.11}
\lefteqn{
\int^{1}_0 \; dy \;e^{-y}\;\big\|e^{-ytS(\gt)} - e^{-ytH}\big\| = }\\
& & \hspace{0.5cm}
\int^1_{\gt/t} \; dy \;e^{-y}\;\big\|e^{-ytS(\gt)} - e^{-ytH}\big\| +
\int^{\gt/t}_0 \; dy \;e^{-y}\;\big\|e^{-ytS(\gt)} - e^{-ytH}\big\|.
\nonumber
\eea
%
%
Since by (\ref{eq:3.3.10}) for $\gt,t,y \in (0,1]$ and $\gt/t \le y$
\bed
\big\|e^{-ytS(\gt)} - e^{-ytH}\big\| \le (1 + c_\gr) \left(\frac{\gt}{ty}\right)^\gr\, ,
\eed
we obtain for $0 < \gt \le t$ the estimate
\be\la{eq:3.3.12}
\int^{1}_{\gt/t} \; dy \;e^{-y}\;\big\|e^{-ytS(\gt)} - e^{-ytH}\big\| \le
(1 + c_\gr) \int^{1}_0 \; dy \; e^{-y}y^{-\gr}\left(\frac{\gt}{t}\right)^\gr .
\ee
Moreover, for $\gr <1 $ one obviously gets
\be\la{eq:3.3.13}
\int^{\gt/t}_0 \; dy \;e^{-y}\;\big\|e^{-ytS(\gt)} - e^{-ytH}\big\|
\le \; 2\left(\frac{\gt}{t}\right)^\gr\, .
\ee
Taking into account (\ref{eq:3.3.12}) and
(\ref{eq:3.3.13}) we obtain from (\ref{eq:3.3.11}):
\be\la{eq:3.3.14}
\int^{1}_0 \; dy \;e^{-y}\;\big\|e^{-ytS(\gt)} - e^{-ytH}\big\| \le
\Big[(1 + c_\gr) \int^{1}_0 \; dy \; e^{-y}y^{-\gr} + 2\Big]
\left(\frac{\gt}{t}\right)^\gr ,
\ee
for $\gt,t \in (0,1]$, with $0 < \gt \le t$.

Finally, by virtue of \eqref{eq:3.3.10} and \eqref{eq:3.3.14} we get for \eqref{eq:3.3.9}
%
%
\bead
\lefteqn{
\|(\mathds{1} +tS(\gt))^{-1} - (\mathds{1} +tH)^{-1}\| \le}\\
& & \hspace{1.0cm}
\sum^\infty_{n=0}\;e^{-n}\Big\{
n\;{{(1 + c_\gr)}} + (1 + c_\gr) \int^{1}_0 \; dy \; e^{-y}y^{-\gr} + 2\Big\}
\left(\frac{\gt}{t}\right)^\gr .
\eead
Now, setting
\begin{equation*}
M_\gr := \sum^\infty_{n=0}\;e^{-n}\Big\{
n\;{{(1 + c_\gr)}} + (1 + c_\gr) \int^{1}_0 \; dy \; e^{-y}y^{-\gr} + 2\Big\}
\end{equation*}
we obtain estimate \eqref{eq:3.3.1}.
%
%
\end{proof}
\begin{remark}\label{rem:3.3.2}
In Lemma \ref{lem:3.3.1}(i) it is shown that for $\gr =1$ the condition \eqref{eq:3.3.1}
implies \eqref{eq:3.3.2}. But it is \textit{unclear} for converse since Lemma \ref{lem:3.3.1}(ii)
does not cover this case.
%
%
%
\end{remark}
The next assertion extends the result of Lemma \ref{lem:3.3.1}(i) to \textit{any} bounded interval
$\cI \subset \dR^{+}_{0}$.
\bt\la{th:3.3.2}
Let $F: \dR^{+}_{0} \longrightarrow \cL(\gotH)$ be a measurable family of non-negative
self-adjoint contractions such that $F(0) = \mathds{1}$. Let self-adjoint family
$\{S(\tau)\}_{\tau > 0}$ be defined by \emph{(\ref{eq:3.1.1}) and (\ref{eq:3.1.2})}
for self-adjoint operator $H \geq 0$ in $\gotH$.

If for some $\gr \in (0,1]$ there is a constant $M_\gr > 0$ such that the estimate \eqref{eq:3.3.1}
holds for $\gt,t \in (0,1]$ and $0 < \gt \le t$, then for any bounded interval $\cI \subset \dR^{+}_{0}$
there is a constant $c^\cI_\gr > 0$ such that the estimate
\be\la{eq:3.3.15}
\sup_{t\in\cI}\|F(t/n)^n - e^{-tH}\| \le c^\cI_\gr \; \frac{1}{n^\gr}\, ,
\ee
holds for $n \geq 1$.
\et
%
\begin{proof}
Let $N \in \dN$ such that $\cI \subseteq [0,N]$. Then representation
\bed
\begin{split}
F(t/Nn)^{Nn}& - e^{-t\, H}\\
& = \sum^{N-1}_{k=0}e^{-k \ t \, H/N}(F(t/Nn)^n - e^{-t\, H/N})F(t/Nn)^{(N-1-k)n} \, ,
\end{split}
\eed
for $n \geq 1$, yields the estimate
\be\la{eq:3.3.151}
\|F(t/Nn)^{Nn} - e^{-t\, H}\| \le N\|F(t/Nn)^n - e^{-t\, H/N}\| .
\ee
Let $t' := t/N \leq 1$ and $\gt' :=  t'/n \leq 1$. Then $t \leq N$ and $0 < \gt' \le t' \le 1$.
By Lemma \ref{lem:3.3.1}(i) we obtain
\bed
\|F(t/Nn)^{n} - e^{-tH/N}\| = \|F(\gt')^{t'/\gt'} - e^{-t'H}\|\le
c_\gr \left(\frac{\tau'}{t'}\right)^\gr ,
\eed
and by (\ref{eq:3.3.151}) the estimate
\bed
\|F(t/Nn)^{Nn} - e^{-tH}\| \le c_\gr N \left(\frac{\tau'}{t'}\right)^\gr.
\eed
Let $n' := Nn \ge 1$. Then
\bed
\|F(t/n')^{n'} - e^{-tH}\| \le c_\gr N^{1+\gr} \left(\frac{1}{n'}\right)^\gr, \quad t \in [0,N].
\eed
Setting $c^{[0,N]}_{\gr} := c_\gr N^{1+\gr}$, we prove the theorem for $\cI = [0,N]$. Since for
any bounded interval $\cI$ one can always find a $N \in \dN$ such that $\cI \subseteq [0,N]$, this
completes the proof.   
\end{proof}
To extend Theorem \ref{th:3.3.2} to $\cI  = \dR^{+}_{0}$ one needs conditions when the values of
$\gt , t$ are allowed to be unbounded.
\bt\la{th:3.3.3}
Let $F: \dR^{+}_{0} \longrightarrow \cL(\gotH)$ be a measurable family of non-negative
self-adjoint contractions such that $F(0) = \mathds{1}$. Let self-adjoint family
$\{S(\tau)\}_{\tau > 0}$ be defined by \emph{(\ref{eq:3.1.1}) and (\ref{eq:3.1.2})}
for self-adjoint operator $H \geq 0$ in $\gotH$.

If for some $\gr \in (0,1]$ there is a constant $M_\gr > 0$ such that the estimate \eqref{eq:3.3.1}
holds for $0 < \gt \le t < \infty$, then there exists a constant $c^{\dR^+}_\gr > 0$ such that
for $\tau =t/n$ the estimate
\be\la{eq:3.3.16}
\sup_{t\in \dR^{+}_{0}}\|F(t/n)^n - e^{-tH}\| \le c^{\dR^+}_\gr \frac{1}{n^\gr} ,
\ee
holds for $n \ge 1$.
\et
%
\begin{proof}
The arguments ensuring that \eqref{eq:3.3.3} yields estimate \eqref{eq:3.3.8} go through verbatim
if we assume \eqref{eq:3.3.1} for $0 < \gt \le t < \infty$. Then setting $\gt := t/n$, $n \in \dN$ we
deduce from \eqref{eq:3.3.8}
\bed
\|F(t/n)^n - e^{-tH}\| \le c^{\dR^+}_\gr \, \frac{1}{n^\gr}\, , \quad n \ge 1,
\eed
for $t \in \dR^{+}_{0}$, where $c^{\dR^+}_\gr := 1 + c \, M_\gr \, $. So, this proves (\ref{eq:3.3.16}).
\end{proof}
\begin{remark}\label{rem:3.3.21}
We \textit{Note} that in Theorem \ref{th:3.2.6} we established the self-adjoint operator-norm convergent
\textit{extension} of the Chernoff product formula for any bounded interval $\cI \subset \dR^{+}_{0}$
under $t$-\textit{dependent} condition (\ref{eq:3.2.19}), which is necessary and sufficient.\\
We \textit{Note} that Theorem \ref{th:3.3.2} and Theorem \ref{th:3.3.3} prove self-adjoint operator-norm
Chernoff product formula in $\dR^{+}_{0}$ with \textit{estimate} of the \textit{rate} of convergence.
They are also based on $t$-\textit{dependent} {fractional power} condition (\ref{eq:3.3.1}),
which is necessary and sufficient for $\rho \in (0,1)$.
\end{remark}
Our next \textit{Note} is that $t$-dependence in assumption (\ref{eq:3.3.1}) for $\gr =1$  can be
relaxed. The assertion below extends to $\dR^{+}_{0}$ (with convergence rate) the established in
Theorem \ref{th:3.2.3} operator-norm convergence of the self-adjoint Chernoff product formula for
bounded interval $\cI \subset \dR^{+}$. It yields an operator-norm version of original Chernoff
product formula, see Proposition \ref{prop:3.1.0}.
\bt\la{th:3.3.4}
Let $F: \dR^{+}_{0} \longrightarrow \cL(\gotH)$ be a measurable family of non-negative
self-adjoint contractions such that $F(0) = \mathds{1}$. Let self-adjoint family
$\{S(\tau)\}_{\tau > 0}$ be defined by \emph{(\ref{eq:3.1.1}) and (\ref{eq:3.1.2})}
for self-adjoint operator $H \geq 0$ in $\gotH$.

If there is a constant $M_1 > 0$ such that the estimate
\be\la{eq:3.3.17}
\|(\mathds{1} +S(\gt))^{-1} - (\mathds{1} +H)^{-1}\| \le M_1\gt \, ,
\ee
holds for $\gt \in (0,1]$, then for any bounded interval $\cI \subset \dR^{+}_{0}$ there is a
constant $c^\cI_1 > 0$ such that the estimate
\be\la{eq:3.3.18}
\sup_{t\in\cI}\|F(t/n)^n - e^{-tH}\| \le c^\cI_1 \, \frac{1}{n}\, ,
\ee
holds for $n \ge 1$.
\et
%
\begin{proof}
By virtue of (\ref{eq:3.2.71}) and \eqref{eq:3.3.17} we obtain the estimate
\be\la{eq:3.3.19}
\begin{split}
\|&(\mathds{1} +tS(\gt))^{-1} - (\mathds{1} +tH)^{-1}\|\\
&\le M_1 \gt \;t \ \|(\mathds{1} +S(\gt))(\mathds{1} +tS(\gt))^{-1}\|\|(\mathds{1}+ H)
(\mathds{1} +tH)^{-1}\|\,.
\end{split}
\ee

Then \eqref{eq:3.3.19}, together with estimates
$\|(\mathds{1} +S(\gt))(\mathds{1} +tS(\gt))^{-1}\| \le {1}/{t}$
and $\|(\mathds{1}+ H)((\mathds{1} +tH)^{-1}\| \le {1}/{t}$, for self-adjoint $S(\gt)$ and $H$,
imply for $0 < \gt \le t \le 1$ and $\gr = 1$ the estimate \eqref{eq:3.3.1} and thus \eqref{eq:3.3.2}.
Finally, applying Theorem \ref{th:3.3.2} we extend the proof to any bounded interval
$\cI \subset \dR^{+}_{0}$.
%
%
\end{proof}
Now we extend Theorem \ref{th:3.3.4} for condition (\ref{eq:3.3.17}), to infinite interval
$\cI = \dR^{+}_{0}$. To this end, similar to Theorem \ref{th:3.3.3}, one needs additional assumption
valid on infinite $t$-intervals.
\bt\la{th:3.3.5}
Let in addition to conditions of Theorem \emph{\ref{th:3.3.4}} operator
$H \ge \mu \mathds{1}$, $\mu > 0$ and for any $\varepsilon > 0$ there exists a $\gd_\varepsilon \in (0,1)$
such that
\be\la{eq:3.3.20}
0 \le F(\gt) \le (1-\gd_\varepsilon)\mathds{1},
\ee
is valid for $\gt \ge \varepsilon$. If there is a constant $M_1 > 0$ such that \eqref{eq:3.3.17}
holds for $\gt \in (0,\varepsilon \leq 1)$, then there exists a constant $c^{\dR^+}_1 > 0$
such that \eqref{eq:3.3.18} is valid for infinite interval $\cI = \dR^{+}_{0}$.
\et
%
\begin{proof}
Since \eqref{eq:3.3.17} implies the resolvent-norm convergence of $\{S(\tau)\}_{\tau > 0}$,
when $\tau \rightarrow  +0$, and since $H \ge \mu \mathds{1}$, there exists
$0 < \mu_{\varepsilon} \le \mu $ such that $S(\gt) \ge \mu_{\varepsilon} \mathds{1}$ for
$\gt \in (0,\varepsilon)$, where $\varepsilon$ is sufficiently \textit{small}.

On the other hand, \eqref{eq:3.3.19} for self-adjoint $S(\gt)$ and $H$, yields
\be\la{eq:3.3.21-1}
\begin{split}
\|&(\mathds{1} +tS(\gt))^{-1} - (\mathds{1} +tH)^{-1}\|= \\
&\le M_1 \, \frac{\gt}{t} \, \|(\mathds{1} +S(\gt))(\mathds{1}/t +
S(\gt))^{-1}\| \|(\mathds{1}+ H)(\mathds{1}/t +H)^{-1}\| \, ,
\end{split}
\ee
for $t > 0$. Since $S(\gt) \ge \mu_{\varepsilon} \mathds{1}$, $\gt \in (0,\varepsilon)$, and
$H \ge \mu \mathds{1}$, we obtain estimates
\bed
\begin{split}
&\|(\mathds{1} +S(\gt))(\mathds{1}/t +S(\gt))^{-1}\| \le \frac{1+\mu_{\varepsilon}}{\mu_{\varepsilon}}\, ,\\
&\|(\mathds{1}+ H)(\mathds{1}/t +H)^{-1}\| \le \frac{1+\mu}{\mu} \, .
\end{split}
\eed
By (\ref{eq:3.3.21-1}) these estimates give
\be\la{eq:3.3.22}
\|(\mathds{1} +tS(\gt))^{-1} - (\mathds{1} +tH)^{-1}\| \le M^{\dR^+}_1 \frac{\gt}{t}\, ,
\ee
for $\gt \in (0,\varepsilon)$ and $t > 0$. Here
$M^{\dR^+}_1 := M_1 \, (1+\mu_{\varepsilon})(1+\mu)/{\mu_{\varepsilon}}{\mu}$.

Note that if $\tau/t \, \leq 1$, then by (\ref{eq:3.3.5})
\bed
\big\|F(\gt)^{t/\gt} - e^{-tS(\gt)}\big\| \le \frac{\gt}{t} \, , \quad  0 < \gt \le t < \infty \, .
\eed
Therefore, \eqref{eq:3.3.3}, \eqref{eq:3.3.6} and
(\ref{eq:3.3.22}), which are valid for $\gt \in (0,\varepsilon \leq 1)$, allow to use the result
(\ref{eq:3.3.18}) of Theorem \ref{th:3.3.4} for the case $0 < \gt \le t < \infty$, and
$\tau = t/n$, $n \ge 1$. This yields
\be\la{eq:3.3.23}
\| F(t/n)^n - e^{-tH}\| \le \widehat{c}^{\ \dR^+}_1 \frac{1}{n} \, ,
\ee
for bounded interval: $t \in [0,\varepsilon n )$, and for
$\widehat{c}^{\ \dR^+}_1 := 1 + c \, M^{\dR^+}_1$.

Now let $t \geq \varepsilon n$. Then by assumption \eqref{eq:3.3.20} we have
\be\la{eq:3.3.24}
\|F(t/n)^n\| \le (1-\gd_\varepsilon)^{n} = e^{n\, \ln(1-\gd_\varepsilon)},
\quad t \geq n\varepsilon \,.
\ee
Note that $H \ge \mu I$ implies $\|e^{-tH}\| \le e^{-n\varepsilon\mu}$ for
$t \geq n\varepsilon$. This together with (\ref{eq:3.3.23}) and (\ref{eq:3.3.24}) yield
for a small $\varepsilon > 0$, cf. (\ref{eq:3.3.22}), the estimate
\bed
\| F(t/n)^n - e^{-tH}\| \le \widehat{c}^{\ \dR^+}_1 \frac{1}{n} +
e^{n\, \ln(1-\gd_\varepsilon)} + e^{-n\varepsilon\mu} \, ,
\eed
valid for \textit{any} $t \geq 0$.

Since $\widetilde{c}_1 := \sup_{n \ge 1} n(e^{n\, \ln(1-\gd_\varepsilon)} +
e^{-n\varepsilon\mu}) < \infty$, there exists constant $c^{\dR^+}_1 :=
\widehat{c}^{\ \dR^+}_1 + \widetilde{c}_1$ such that \eqref{eq:3.3.18} is valid for $n \ge 1$
and infinite interval $\cI = \dR^{+}_{0}$
%
%
\end{proof}
\section{{Nonself-adjoint operator-norm Chernoff product formula}}\la{sec:3.4}
The results on the \textit{nonself-adjoint} Chernoff product formula in operator-norm topology
are more restricted. The most of them concern the \textit{quasi-sectorial} contractions \cite{CZ01}.
\begin{definition}\label{def:2.1.2}
A {contraction} $F$ on the Hilbert space $\mathfrak{H}$ is called
\textit{quasi-sectorial} with  semi-angle $\alpha\in [0, \pi/2)$
with respect to the vertex at $z=1$, if its numerical range $W(F)\subseteq D_{\alpha}$, where
closed domain
\begin{equation}\label{eq:2.1.1}
 D_{\alpha}:=\{z\in {\mathbb{C}}: |z|\leq \sin \alpha\} \cup
\{z\in {\mathbb{C}}: |\arg (1-z)|\leq \alpha \ {\rm{and}}\ |z-1|\leq
\cos \alpha \} .
\end{equation}
The limits: $\alpha=0$ and $\alpha = \pi/2-0$, correspond, respectively, to non-negative self-adjoint
contractions and to general contractions.
\end{definition}
A characterisation of {quasi-sectorial} contraction \textit{semigroups} is due to \cite{Zag08},
\cite{ArZ10}.
\begin{proposition}\label{prop:6.2.2}
$C_0$-semigroup $\{e^{- t \, H}\}_{t \geq 0}$ is, for $t>0$, a family of quasi-sectorial contractions
with $W(e^{-t \, H}) \subseteq D_{\alpha} \, $, if and only if generator $H$ is an $m$-sectorial
operator with $W(H) \subset S_{\alpha}$, the open sector with semi-angle $\alpha\in[0, \pi/2)$ and
vertex at $z=0 \,$.
\end{proposition}
Note that if operator $F$ is a quasi-sectorial contraction and $W(F)\subseteq D_{\alpha}$, then
$\mathds{1}- F$ is also $m$-sectorial operator with vertex $z=0$ and semi-angle $\alpha$. Using the
Riesz-Dunford functional calculus one obtains estimate
%
%
\begin{equation}\label{eq:2.1.14}
\|F^n (\mathds{1}-F)\|\leq \frac{K}{n+1} \ , \ n\in{\mathbb{N}} \ .
\end{equation}
The estimate (\ref{eq:2.1.14}) allows to go beyond the Chernoff $\sqrt{n}$-Lemma (\ref{eq:0.5}) and
to establish the $(1/\sqrt[3]{n})$-Theorem  \cite{Zag17}.
\begin{proposition}\label{th:6.2.2}
Let $F$ be a quasi-sectorial contraction on $\mathfrak{H}$ with
numerical range $W(F)\subseteq D_\alpha$ for $\alpha\in [0, \pi/2)$. Then
\begin{equation}\label{eq:6.2.5}
\left\|F^n - e^{n(F-\mathds{1})}\right\| \leq {M\over n^{1/3}} \ , \ \ n\in{\mathbb{N}}\, ,
\end{equation}
where $M=2K+2$ and $K$ is defined by {\rm{(\ref{eq:2.1.14})}}.
\end{proposition}

Next we recall nonself-adjoint operator-norm extension of the Trotter-Neveu-Kato
convergence theorem for \textit{quasi-sectorial} contraction semigroups \cite{CZ01} (Lemma 4.1):
\begin{proposition}\label{th:6.2.21}
Let $\{S(\tau)\}_{\tau >0}$ be a family of $m$-sectorial operators with $W(S(\tau))\subseteq S_\alpha$
for some $\alpha \in [0, \pi/2)$ and for all $\tau>0$. Let $H$ be an {$m$-sectorial} operator with
$W(H) \subset S_{\alpha}$.
Then the following conditions are equivalent:
\begin{eqnarray*}
(a) & & \lim_{\tau\rightarrow +0} \left\|(\zeta \mathds{1} + S(\tau))^{-1} -
(\zeta \mathds{1} + H)^{-1}\right\| = 0, \ \
\mbox{ for some } \zeta\in S_{\pi-\alpha};\\
(b) & & \lim_{\tau\rightarrow +0} \left\|e^{-t\, S(\tau)} - e^{-t\, H} \right\| = 0, \ \
\mbox{ for t in a subset of \ } \dR^{+} \mbox{ having a limit point}.
\end{eqnarray*}
\end{proposition}

Therefore, the estimate (\ref{eq:6.2.5}) together with Proposition \ref{th:6.2.21} and inequality
\begin{equation}\label{eq:6.2.51}
\|F(t/n)^n - e^{-t\, H}\| \leq \|F(t/n)^n - e^{-t S(t/n)}\| +  \|e^{-tS(t/n)} - e^{-tH}\| \, ,
\end{equation}
yield \textit{nonself-adjoint} operator-norm version of the Chernoff product
formula for quasi-sectorial contractions (cf. Proposition \ref{prop:3.1.0} in \cite{CZ01}):
\begin{proposition}\label{prop:2.1.12}
Let $\{F(\tau)\}_{\tau\geq 0}$ be a family of uniformly
quasi-sectorial contractions on a Hilbert space $\mathfrak{H}$, that is, $W(F(\tau) \subseteq D_\alpha$
{\rm{(\ref{eq:2.1.1})}}, for all $\tau > 0$. Let family $\{S(\tau)\}_{\tau > 0}$ be defined by
\emph{(\ref{eq:3.1.1}) and (\ref{eq:3.1.2})}. If $H$ is {$m$-sectorial} operator with
$W(H) \subset S_{\alpha}$, then
\begin{equation}\label{eq:2.1.17}
\lim_{n\rightarrow \infty} \left\|F(t/n)^n -e^{-t\, H}\right\| = 0 \, , \ \ \ \mbox{for} \ \ t> 0 \ ,
\end{equation}
if and only if
\begin{equation}\label{eq:2.1.171}
\lim_{\tau\rightarrow +0} \left\|(\zeta \mathds{1} + S(\tau))^{-1} -
(\zeta \mathds{1} + H)^{-1}\right\| = 0 \, , \ \ \ \mbox{ for some } \zeta\in S_{\pi-\alpha}\, .
\end{equation}
\end{proposition}
\begin{remark}\label{rem:3.3.3}
Since semigroups $\{e^{- z \, S(\tau)}\}_{z \in S_{\pi/2-\alpha}}$, $\tau > 0$, and
$\{e^{- z \, H}\}_{z \in S_{\pi/2-\alpha}}$ are holomorphic in sector $S_{\pi/2-\alpha}$,
proof of the Chernoff product formula (\ref{eq:2.1.17}) for {nonself-adjoint}
quasi-sectorial contractions is based on the Riesz-Dunford functional calculus.
To establish the operator-norm convergence \textit{without} rate it
successfully (although \textit{not} completely, since $t > 0$) substitutes the self-adjointness
in the proofs in Sections \ref{sec:3.2} and \ref{sec:3.3}.
\end{remark}

We use this calculus to prove the operator-norm Chernoff product formula for
quasi-sectorial contractions \textit{with} estimate of the {rate} of convergence.
We \textit{Note} that this is a generalisation
of the Chernoff product formulae proven respectively in Theorem \ref{th:3.3.4}
(for self-adjoint case) and in Proposition \ref{prop:2.1.12} (without rate of convergence).
\begin{theorem}\label{th:6.2.22}
Let $\{F(\tau)\}_{\tau\geq 0}$ be a family of uniformly quasi-sectorial contractions on a Hilbert
space $\mathfrak{H}$ and let $\{S(\tau)\}_{\tau >0}$ be a family of $m$-sectorial operators
defined by \emph{(\ref{eq:3.1.1}) and (\ref{eq:3.1.2})}
for {$m$-sectorial} operator $H$ with $W(H)\subseteq S_\alpha$.
If there is  $L > 0$ such that estimate
\begin{equation}\label{est-res}
\left\|(\zeta \mathds{1} + S(\tau))^{-1} -
(\zeta \mathds{1} + H)^{-1}\right\| \leq { L \ \tau \over {\rm{dist}}(\zeta,-S_\alpha)}\, ,
\end{equation}
holds for $\zeta \in S_{\pi-\alpha} \ $, then for any bounded interval $\cI \subset \dR^{+}$ there is a
constant $C^\cI > 0$ such that estimate
\begin{equation}\label{est-Ch}
\sup_{t\in\cI}\|F(t/n)^n - e^{-tH}\| \le C^\cI \, \frac{1}{n^{1/3}} \ ,
\end{equation}
holds for $n \ge 1$.
\end{theorem}
\begin{proof}
Estimation of the last term in inequality (\ref{eq:6.2.51}). Since by (\ref{eq:3.1.1})
and conditions on $\{F(\tau)\}_{\tau\geq 0}$ operators $\{S(\tau)\}_{\tau>0}$ are $m$-sectorial
with $W(S(\tau))\subseteq S_\alpha$, the Riesz-Dunford formula
\begin{equation*}
e^{-tS(\tau)} = {1\over 2\pi i} \int_\Gamma d\zeta \ {e^{t\zeta}\over\zeta \mathds{1} + S(\tau)} \, ,
\end{equation*}
defines for $\tau>0$ a family of holomorphic semigroups
$\tau \mapsto \{e^{-tS(\tau)}\}_{t\in S_{\pi/2-\alpha}}$.
Here $\Gamma\subset S_{\pi-\alpha}$ is a positively-oriented closed (at infinity) contour in $\dC$
around $-S_\alpha$.
The same is true for $m$-sectorial operator $H$ since $W(H)\subseteq S_\alpha$:
\begin{equation*}
e^{-tH} = {1\over 2\pi i}\int_\Gamma d\zeta \ {e^{t\zeta}\over \zeta \mathds{1} + H}\, .
\end{equation*}

We define $\Gamma := \overline{\Gamma_\epsilon} \cup\Gamma_\delta \cup \Gamma_\epsilon$,
where the arc $\Gamma_\delta = \{z\in{\C}: \ |z|=\delta>0, |\arg z|\leq \pi-\alpha-\epsilon\}$
(for $0<\epsilon<\pi/2-\alpha$) and $\Gamma_\epsilon,\overline{\Gamma_\epsilon}$ are two conjugate
radial rays with $\Gamma_\epsilon = \{z\in{\C}: \ \arg z = \pi-\alpha-\epsilon,\ |z|\geq \delta\}$.
Then for $t>0$ one gets

\begin{equation}\label{esa1}
\| e^{-tS(\tau)} - e^{-tH}\| \leq {1\over 2\pi} \int_\Gamma |d\zeta||e^{t\zeta}|
\left\|(\zeta \mathds{1}+S(\tau))^{-1} - (\zeta \mathds{1} +H)^{-1}\right\|.
\end{equation}

Since operators $\{S(\tau)\}_{\tau>0}$ and $H$ are $m$-sectorial with numerical ranges in open sector
$S_\alpha$, by condition (\ref{est-res}) we obtain:
\begin{eqnarray}\label{est1}
&&\left\|(\zeta \mathds{1}+S(\tau))^{-1} - (\zeta \mathds{1}+H)^{-1}\right\| \leq
{ L \, \tau \over \delta \, \sin \varepsilon}\, , \  \ \ \rm{for} \ \zeta \in \Gamma_\delta \ , \\
&&\left\|(\zeta \mathds{1}+S(\tau))^{-1} - (\zeta \mathds{1}+H)^{-1}\right\| \leq
{ L \, \tau \over |\zeta|\, \sin \varepsilon}\, , \  \ \rm{for} \
\zeta \in \Gamma_\epsilon \vee \overline{\Gamma_\epsilon} \ .
\label{est2}
\end{eqnarray}
Then for $t>0$ the estimates (\ref{esa1}) and (\ref{est1}), (\ref{est2}) yield
\begin{equation}\label{esa2}
\| e^{-tS(\tau)} - e^{-tH}\| \leq
\frac{L \tau}{\sin \varepsilon} \, \left\{e^{t \delta } +
\frac{e^{- t \delta \,\cos (\alpha + \varepsilon)}}{\pi \, t \, \delta \,
\cos (\alpha + \varepsilon)}\right\}\, .
\end{equation}
Hence, inequality (\ref{esa2}) proves the existence of $N^\cI$ such that for $\tau = t/n$
\begin{equation}\label{esa3}
\sup_{t\in\cI}\| e^{-tS(t/n)} - e^{-tH}\| \leq N^\cI \ \frac{1}{n} \, .
\end{equation}

To estimate the first term in the right-hand side of (\ref{eq:6.2.51}) we use Proposition \ref{th:6.2.2}.
Then
\begin{equation}\label{esa4}
\|F(t/n)^n - e^{-t S(t/n)}\|  \leq M \ {1 \over n^{1/3}} \ .
\end{equation}
The inequalities (\ref{eq:6.2.51}) and (\ref{esa3}), (\ref{esa4}) prove the estimate (\ref{est-Ch}).
\end{proof}
\begin{remark}\label{rem:6.2.23}
Note that the rate (\ref{est-Ch}) of the operator-norm convergence of the Chernoff product formula for
quasi-sectorial contractions is slower then for the self-adjoint case (\ref{eq:3.3.18}). This rate is
limited by \textit{non-optimal} estimate due to Proposition \ref{th:6.2.2}.
\end{remark}
\begin{corollary}\label{cor:6.2.24}
Let $\{S(\tau)\}_{\tau >0}$ be a family of $m$-sectorial operators
defined by \emph{(\ref{eq:3.1.1}) and (\ref{eq:3.1.2})} with $W(S(\tau))\subseteq S_\alpha$
for some $\alpha \in [0, \pi/2)$.
Let $H$ be an {$m$-sectorial} operator with $W(H)\subseteq S_\alpha$. Then there exist $M' > 0$ and
$\delta' > 0$ such that
\begin{equation}\label{esa5}
\| e^{-tS(\tau)} - e^{-tH}\| \leq
\frac{M'}{t} \ e^{t \, \delta'} \,
\left\|(\mathds{1}+S(\tau))^{-1} - (\mathds{1}+H)^{-1}\right\| ,
\end{equation}
holds for positive $\tau$ and $t$.
\end{corollary}

For extension of Theorem \ref{th:6.2.22} to $\dR^{+}_{0}$ the estimate (\ref{esa5}) suggests a
weaker form of the operator-norm Trotter-Neveu-Kato theorem.
We \textit{Note} that, under \textit{stronger} than (\ref{est-res}) (the $t$-dependent resolvent condition,
see (\ref{eq:3.3.1}) for $\rho = 1$)
we obtain a new version of Proposition \ref{th:6.2.21}.
\begin{theorem}\label{th:6.2.25}
Let $\{F(\tau)\}_{\tau\geq 0}$ be a family of uniformly quasi-sectorial contractions on a
Hilbert space $\mathfrak{H}$. Let $\{S(\tau)\}_{\tau >0}$ be a family of $m$-sectorial operators
defined by \emph{(\ref{eq:3.1.1}), (\ref{eq:3.1.2})} with $W(S(\tau))\subseteq S_\alpha$
for some $\alpha \in [0, \pi/2)$ and for {$m$-sectorial} operator $H$ with
$W(H)\subseteq S_\alpha$. Then estimate
\be\la{eq:3.2.161}
\sup_{t\in \cI} \, \left\|(\zeta \mathds{1} + t S(\tau))^{-1} -
(\zeta \mathds{1} + t H)^{-1}\right\| \leq { L^{\cI}  \tau \over {\rm{dist}}(\zeta,-S_\alpha)}\, , \ \ \ \
\zeta \in S_{\pi-\alpha}\, ,
\ee
holds for any interval $\cI \subseteq \dR^{+}_{0}$ if and only if the condition
\be\la{eq:3.2.162}
\sup_{t\in \cI} \, \|e^{-tS(\tau)} - e^{-tH}\| \leq K^{\cI}  \tau \ ,
\ee
is valid for any interval $\cI \subseteq \dR^{+}_{0}$.
\end{theorem}
\begin{proof}
Necessity of (\ref{eq:3.2.162}): As in the proof of Theorem \ref{th:6.2.22} we use the Riesz-Dunford
functional calculus for holomorphic semigroups to obtain estimate
\begin{equation}\label{esa11}
\| e^{-tS(\tau)} - e^{-tH}\| \leq {1\over 2\pi} \int_\Gamma |d z||e^{z}|
\left\|(z \mathds{1}+t\,S(\tau))^{-1} - (z \mathds{1} +t\,H)^{-1}\right\| \, ,
\end{equation}
for the same contour $\Gamma \subset S_{\pi-\alpha}$. After change of variable: $z = t \, \zeta$,
the right-hand side of estimate (\ref{esa11}) gets the same expression as (\ref{esa2}), but for $\delta$
substituted by $\delta/t$. This yields (\ref{eq:3.2.162}) for any bounded interval $\cI \subset \dR^{+}_{0}$.

Sufficiency of (\ref{eq:3.2.162}): First, using the Laplace transform, we estimate:
\begin{equation}\label{esa21}
\left\|(\zeta \mathds{1}+t\,S(\tau))^{-1} - (\zeta \mathds{1} +t\,H)^{-1}\right\|
\leq \int^\infty_0 ds \ e^{-s \, \RE\zeta} \big\|e^{-s \, t \, S(\tau)} - e^{-s \, t \, H}\big\|\, ,
\end{equation}
in the half-plane $\dC_{+} = S_{\pi/2}$. To this aim we
let $\varepsilon \in (0,1)$ and  $N_{\varepsilon} := - \ln(\varepsilon /2)$ such that for
$\zeta \in S_{\pi/2}$ and $\tau > 0$, $t \ge 0$
\begin{equation*}
\int^{\infty}_{N_{\varepsilon}} ds \ e^{- s \, \RE\zeta}\left\|e^{-s \, t \, S(\tau)} -
e^{- s \, t \, H}\right\| \le {\varepsilon} \, .
\end{equation*}
Therefore, one gets
\begin{equation*}
\left\|(\zeta \mathds{1}+t\,S(\tau))^{-1} - (\zeta \mathds{1} +t\,H)^{-1}\right\| \le
\int^{N_{\varepsilon}}_{0} ds \ e^{-s \, \RE\zeta}\left\|e^{-s\,t \,S(\tau)} - e^{-s\,t \,H}\right\| +
{\varepsilon} \ ,
\end{equation*}
that by condition (\ref{eq:3.2.162}) for any interval $\cI \subseteq \dR^{+}_{0}$ and
$\varepsilon \in (0,1)$ yields
\begin{eqnarray}\label{esa31}
&&\sup_{t\in\cI}\left\|(\zeta \mathds{1} +tS(\tau))^{-1} - (\zeta \mathds{1} +tH)^{-1}\right\| \le \\
&& \frac{1}{\RE\zeta}\sup_{\ba{c} t\in\cI \wedge s\in [0,N_{\varepsilon}]\ea}\left\|e^{-s\,tS(\tau)} -
e^{-s\,tH}\right\| + {\varepsilon} \leq  \frac{1}{\RE\zeta} \ K^{\dR^{+}_{0}}  \tau  + {\varepsilon} \ .
\nonumber
\end{eqnarray}
Since $\varepsilon$ may be arbitrary small, we obtain the estimate (\ref{eq:3.2.161}) for
any $\zeta \in S_{\pi/2}$.

For extension of $\zeta$ to sector $S_{\pi-\alpha}$ we note that semigroups involved
into estimate (\ref{esa21}) are holomorphic in sector $S_{\pi/2-\alpha}$. Therefore, the Laplace transform
is also valid for integration along the radial rays:  $\, s \, e^{i \varphi} \in S_{\pi/2-\alpha}\, $. Then
conditions for convergence of Laplace integrals take the form:
\begin{equation}\label{esa41}
- {\pi}/{2}< \varphi + \arg\zeta < {\pi}/{2} \ \ \  \ \wedge \ \ \
-(\pi/2-\alpha) < \varphi < (\pi/2-\alpha) \, ,
\end{equation}
that yields $\arg\zeta \in S_{\pi-\alpha}$ and makes $\RE(e^{i \varphi} \zeta)$ proportional to
${\rm{dist}}(\zeta,-S_\alpha)$.
\end{proof}
\begin{corollary}\label{cor:6.2.26}
Linear $\tau$-estimate {\rm{(\ref{eq:3.2.162})}} implies
\be\la{eq:3.2.163}
\sup_{t\in \cI} \, \|e^{-tS(t/n)} - e^{-tH}\| \leq K^{\cI}_{1} \frac{1}{n} \ ,
\ee
for any $n \in \mathbb{N}$ and any bounded interval $\cI \subset \dR^{+}_{0}$.
Then the same line of reasoning as in {\rm{Theorem \ref{th:6.2.22}}} yields
under $t$-dependent resolvent condition {\rm{(\ref{eq:3.2.161})}} the estimate
\begin{equation}\label{est-Ch1}
\sup_{t\in\cI}\|F(t/n)^n - e^{-tH}\| \le C^\cI_{1} \, \frac{1}{n^{1/3}} \ ,
\end{equation}
for $n \ge 1$ and any bounded interval $\cI \subset \dR^{+}_{0}$.
\end{corollary}
We \textit{Note} that extension (\ref{est-Ch1}) of the operator-norm Chernoff product formula
for quasi-sectorial contractions on $\dR^{+}_{0}$ inherits the estimate of the {rate} of convergence
established in Proposition \ref{th:6.2.2}.

\section{Comments: Trotter-Kato product formulae}\label{sec:3.5}
The first application of the Chernoff product formula (Proposition \ref{prop:3.1.0})
was the proof of the strongly convergent Trotter product formula, see \cite{Che68}:
\begin{equation}\label{Trot}
\slim_{n\to\infty} \big(e^{-tA/n}e^{-tB/n}\big)^n = e^{-tH} \, , \ \ \ t\geq 0.
\end{equation}
Here $A$ and $B$ are positive self-adjoint operators in a Hilbert space $\mathfrak{H}$
with domains $\dom\, A$ and $\dom\,B$ such that
$\dom\,A \cap \dom\,B = {\rm{core}}\, H$ of the self-adjoint $H$. Note that operator family
$\{F(t) := e^{-tA}e^{-tB}\}_{t\geq 0}$ is not self-adjoint.

Later the operator-norm Chernoff product formula from Sections \ref{sec:3.2} and \ref{sec:3.3}
was used to lift (\ref{Trot}) to operator-norm topology, as well as to extend it from the
\textit{exponential} Trotter product formula to the \textit{Trotter-Kato} product formulae
for Kato functions $\mathcal{K}$, see \cite{NZ98}.

Recall that if a real-valued Borel measurable function $f: [0, \infty) \rightarrow [0,1]$ satisfying
\begin{equation}\label{eq:K-F}
  0 \leq f(s) \leq 1, \quad f(0) = 1, \quad f'(+0) = -1 \, ,
\end{equation}
then $f \in \mathcal{K}$. Different conditions on local continuity at $t = +0$ and global behaviour
on $\dR^{+}$ select subclasses of the Kato functions, see Appendix C in \cite{Zag19}.
If functions $f,g \in \mathcal{K}$ substitute exponents in formula (\ref{Trot}) then it is called the
{Trotter-Kato} product formula.

To apply a full power of the self-adjoint Chernoff product formula (Sections \ref{sec:3.2} and \ref{sec:3.3})
we symmetrise and produce \textit{self-adjoint} family
$\{F(t) := g(tB)^{1/2}f(tA)g(tB)^{1/2}\}_{t\geq 0}$. Let positive operators
$A$ and $B$ be such that operator $A+B =: H \geq \mu \mathds{1}$ is self-adjoint. If $F(t)$ is
sufficiently smooth at $t = +0$ and satisfy (\ref{eq:3.3.20}) (see \cite{IT01} (1.2)), then (\ref{eq:3.3.17})
holds for $\tau = t/n$ and Theorem \ref{th:3.3.5} proves the operator-norm convergent symmetrised
{Trotter-Kato} product formula:
\begin{equation}\label{T-K}
\|\cdot\|-{\lim_{n\to\infty}} \big(g(tB/n)^{1/2}f(tA/n)g(tB/n)^{1/2}\big)^n = e^{-tH} \, ,
\end{equation}
uniformly on $\dR^{+}_{0}$ with $O(1/n)$ as the rate of convergence, see \cite{IT01}, \cite{ITTZ01}.
There it was also shown that this rate is \textit{optimal}.

For proving convergence of the \textit{nonself-adjoint} Trotter-Kato approximants, for example
the simplest: $\{(f(tA/n)g(tB/n))^n\}_{n \geq 1}$, note that for $n \in \dN$ and $t \geq 0$:
\bed
(f(tA/n)g(tB/n))^n = f(tA/n)g(tB/n)^{1/2}F(t/n)^{n-1}g(tB/n)^{1/2} \ .
\eed
This representation yields:
\bed
\begin{split}
\|(f(tA/n)&g(tB/n))^n - e^{-tH}\| \le \|F(t/n)^{n-1} - e^{-tH}\| \\
&+ 2\|(\mathds{1}- g(tB/n))e^{-tH}\| + \|(\mathds{1}- f(tA/n))e^{-tH}\| \ .
\end{split}
\eed
Then by estimates (\ref{eq:3.3.23}) and (\ref{eq:2.1.14}) we get for $n-1 \geq 1$
\begin{equation}\label{eq:3.4.1}
\| F(t/n)^{n-1} - e^{-tH}\| \le ({\widehat{c}}^{\ \dR^+}_1 + K) \ \frac{1}{n} \ , \
\ t\geq 0.
\end{equation}
On the other hand, since $f,g \in {{\mathcal{K}}}$ and $H = A + B$:
\begin{equation}\label{eq:3.4.2}
\|(\mathds{1}- f(t A/n))e^{-t H}\|\leq C^A_C \ \gamma[f]\ \frac{1}{n} \quad \mbox{and}
\quad \|(\mathds{1}- g(t B/n))e^{-t H}\|\leq C^B_C \ \gamma[g] \ \frac{1}{n} \ ,
\end{equation}
where $\gamma[f]: = \sup_{x > 0}(1 - f(x))/{x}$ and similar for $g$.
Inequalities \eqref{eq:3.4.1} and \eqref{eq:3.4.2} yield for some $\Gamma >0$ the estimate
\bed
\|(f(tA/n)g(tB/n))^n - e^{-tC}\| \le \Gamma \ \frac{1}{n} \ ,
\eed
which proves the asymptotic $O(1/n)$ for $n \rightarrow \infty $.

\bigskip
\textit{This paper is dedicated to the memory of Hagen Neidhardt passed away on 23 March 2019.
I am deeply grateful to Hagen for valuable discussions on the subjects of this and of many
others of my projects.}


\end{document}